\documentclass[a4paper]{amsart}

\usepackage{color}
\usepackage{amsxtra}
\usepackage{mathtools}
\usepackage{enumerate}
\usepackage{nicefrac}
\mathtoolsset{showonlyrefs} 
\usepackage{ mathrsfs }
\usepackage{tikz}
\usetikzlibrary{shadings}
\usepackage{color}
\usepackage{hyperref}
\hypersetup{
  colorlinks   = true, 
  urlcolor     = blue, 
  linkcolor    = blue, 
  citecolor   = red 
}

\DeclareMathOperator*{\esssup}{ess\,sup}
\DeclareMathOperator*{\essinf}{ess\,inf}

\newtheorem{theorem}{Theorem}[section]

\newtheorem{remark}{Remark}[section]

\title{Traces for
fractional Sobolev spaces with variable exponents
}

\author[L. M. Del Pezzo]{Leandro M. Del Pezzo}
	\address{Leandro M. Del Pezzo \hfill\break\indent
		CONICET and 
		\hfill\break\indent
		Departamento de Matem\'aticas y 
		Estad\'istica
		\hfill\break\indent Universidad Torcuato Di Tella
		\hfill\break\indent Av. Figueroa Alcorta 7350 (C1428BCW)
		\hfill\break\indent Buenos Aires, Argentina. }
	\email{ldelpezzo@utdt.edu}
	\urladdr{http://cms.dm.uba.ar/Members/ldpezzo/}
	
\author[J. D. Rossi]{Julio D. Rossi}
    \address{Julio D. Rossi\hfill\break\indent
        CONICET and 
        \hfill\break\indent
        Departamento  de Matem{\'a}tica, FCEyN,
        \hfill\break\indent
        Universidad de Buenos Aires,
        \hfill\break\indent Pabellon I, Ciudad Universitaria (1428),
        \hfill\break\indent
        Buenos Aires, Argentina.}
\email{jrossi@dm.uba.ar}
\urladdr{http://mate.dm.uba.ar/$\sim$jrossi/}

\keywords{$p-$Laplacian, fractional operators, variable exponents. \\
\indent AMS-Subj Class: 	46E35, 45G10, 45P05, }

\begin{document}

\begin{abstract}
    In this note we prove a trace theorem in fractional spaces with 
    variable exponents. To be more precise, we show that if 
    $p\colon\overline{\Omega }\times \overline{\Omega }
    \rightarrow (1,\infty )$ and $q\colon\partial \Omega 
    \rightarrow (1,\infty )$ are continuous functions such 
    that
    \[
        \frac{(n-1)p(x,x)}{n-sp(x,x)}>q(x)
        \qquad \mbox{ in } \partial \Omega \cap 
        \{x\in\overline{\Omega}\colon n-sp(x,x) >0\},
    \]
   then the inequality
    $$
        \Vert f\Vert _{\scriptstyle  L^{q(\cdot)}(\partial \Omega )}
        \leq C \left\{ \Vert f\Vert _{\scriptstyle L^{\bar{p}(\cdot)}(\Omega )}+
        [f]_{s,p(\cdot,\cdot)} \right\}
    $$
    holds. 
    Here $\bar{p}(x)=p(x,x)$ and $\lbrack f\rbrack_{s,p(\cdot,\cdot)} $ denotes 
    the fractional seminorm with variable exponent,
    that is given by
    \begin{equation*}
        \lbrack f\rbrack_{s,p(\cdot,\cdot)} \coloneqq 
        \inf \left\{ \lambda >0\colon
        \int_{\Omega}\int_{\Omega }\frac{|f(x)-f(y)|^{p(x,y)}}{\lambda ^{p(x,y)}
        |x-y|^{n+sp(x,y)}}dxdy<1\right\}
    \end{equation*}
    and $\Vert f\Vert _{\scriptstyle L^{q(\cdot)}(\partial \Omega )}$ and 
    $\Vert f\Vert _{\scriptstyle L^{\bar{p}(\cdot)}(\Omega )}$ are the usual
    Lebesgue norms with variable exponent.
\end{abstract}

\maketitle

\section{Introduction}
    
    We begin this article remembering the definition of the variable exponents 
    Lebesgue space, to this end we follow \cite{diening}.
    Let $(A, \Sigma,\mu)$ be a $\sigma-$finite complete measure space. Then by $ \mathcal{M}(A,\mu)$
    we denote the space of all $f\colon A\to [-\infty,+\infty]$ $\mu-$measurable
    functions. We say that $p\in \mathcal{M}(A,\mu)$ is a bounded variable exponent if
    \[  
        1\le p_-\coloneqq\essinf_{x\in A} p(x)\le p_+\coloneqq\esssup_{x\in A} p(x)<+\infty.
    \] 
    Then the variable exponent Lebesgue space $L^{p(\cdot)}(A,\mu)$ is defined as
    \[
	    L^{p(\cdot)}(A,\mu)
	    \coloneqq\left\{
	    f\in\mathcal{M}(A,\mu)\colon \exists\lambda>0\text{ such that }
	    \int_{A} \left|\frac{f(x)}{\lambda} \right|^{r(x)}d\mu <+\infty
	    \right\}
    \]
    equipped with the norm 
    \[
        \Vert f\Vert_{\scriptstyle L^{p(\cdot)}(A,\mu)}\coloneqq\inf 
        \left\{ \lambda >0\colon\int_{A}
        \left(\frac{|f(x)|}{\lambda}\right)^{p(x)} d\mu(x)<1
        \right\}.
    \]
    In the special case that $\mu_1$ is the $n-$Lebsgue measure, 
    $\mu_2$ is the $(n-1)-$Huassdorff measure, $\Omega$ is a smooth bounded 
    domain of $\mathbb{R}^n,$  $\Sigma_1$  is the $\sigma-$algebra of
    $\mu_1-$measurable set of $\Omega,$  $\Sigma_2$  is the $\sigma-$algebra of
    $\mu_2-$measurable set of $\partial\Omega,$ $p\in\mathcal{M}(\Omega,\mu_1)$
    and $q\in\mathcal{M}(\Omega,\mu_2)$ are bounded variable exponents,  
    we note $L^{p(\cdot)}(\Omega)\coloneqq L^{p(\cdot)}(\Omega,\mu_1)$ and 
    $L^{q(\cdot)}(\partial\Omega)\coloneqq L^{q(\cdot)}(\partial\Omega,\mu_2).$

    From now on let $\Omega$ be a fixed smooth bounded 
    domain in $\mathbb{R}^n$.  Let $p$ be a bounded variable exponent in 
    $\overline{\Omega}\times \overline{\Omega}$, $\bar{p}(x)\coloneqq p(x,x)$
    and  $0<s<1.$  We now introduce the variable exponent 
    Sobolev fractional space as follows:
    \begin{align*}
	    \displaystyle 
	     &W^{s,p(\cdot,\cdot)}(\Omega )\coloneqq\\
	     &\left\{f\in L^{\bar{p}(\cdot)}(\Omega )\colon
        \displaystyle\int_{\Omega}\int_{\Omega }
        \frac{|f(x)-f(y)|^{p(x,y)}}{\lambda ^{p(x,y)}|x-y|^{n+sp(x,y)}}dxdy<+\infty 
        \text{ for some }\lambda >0
        \right\},
    \end{align*}
    and we set 
    \begin{equation*}
        \lbrack f\rbrack_{s,p(\cdot,\cdot)}(\Omega)\coloneqq\inf 
        \left\{ \lambda >0\colon\int_{\Omega}\int_{\Omega }
        \frac{|f(x)-f(y)|^{p(x,y)}}{\lambda ^{p(x,y)}|x-y|^{n+sp(x,y)}}dxdy<1
        \right\}
    \end{equation*}
    as the variable exponent seminorm. When there is
    no confusion we omit the set $\Omega$ from the notation.
    
    It is easy to see that $W^{s,p(\cdot,\cdot)}(\Omega )$ is a Banach space with the norm 
    \[
	    \Vert f\Vert _{s,p(\cdot,\cdot)}\coloneqq
	    \Vert f\Vert _{\scriptstyle L^{\bar p(\cdot)}(\Omega )}
	    +\lbrack f\rbrack_{s,p(\cdot,\cdot)}.   
     \]
   
    To show this fact, one just has to follow the arguments in \cite{Hich} for 
    the constant exponent case. This space $W^{s,p(\cdot,\cdot)}(\Omega )$ was 
    recently introduced in \cite{KRV}. 
    For general theory of classical Sobolev spaces we
    refer the reader to \cite{AF,DD} and for the variable exponent case to \cite{diening}.
    From an applied point of view we recall that non-local energies 
    with constant exponents (we quote here \cite{osher1,osher2}) and 
    also local equations with variable exponents (see \cite{osher3}) where used in image 
    processing. The space $W^{s,p(\cdot,\cdot)}(\Omega )$ 
    defined above combines the two features, it is given by a fractional seminorm 
    with a variable exponent.

    Now we consider two continuous
    variable exponents, one defined in $\overline{\Omega} \times \overline{\Omega}$
    (that was used to define the previous space $W^{s,p(\cdot,\cdot)}(\Omega )$)
    and the other on $\partial \Omega$ (that is used for the usual Lebesgue space
    $L^{\scriptstyle q(\cdot)}(\partial \Omega )$). We assume that
    both $p$ and $q$ are bounded away from $1$ and $\infty,$ that is,
   \[
        1<p_{-}\leq p_{+}< +\infty
        \quad\text{and}\quad  1<q_{-}\leq q_{+}< +\infty  . 
    \]

    Our main result in this note is the following compact embedding trace theorem 
    into variable exponent Lebesgue spaces.

    \begin{theorem}\label{sobolev.trace}  
        If $1<sp_- $ and
        \begin{equation} \label{cota.q}
            p^\star(x)\coloneqq\frac{(n-1)\bar{p}(x)}{n-s\bar{p}(x)}>q(x) \quad \text{ in }
            \partial \Omega \cap \{x\in\overline{\Omega}\colon n-s\bar{p}(x) >0\},
        \end{equation}
        then there is a constant 
        $C=C(n,s,p,q,\Omega )$ such that  
        \begin{equation*}
                \Vert f\Vert _{L^{\scriptstyle q(\cdot)}(\partial \Omega )}\leq 
                C\Vert f\Vert _{s,p(\cdot,\cdot)}
                \quad \forall f\in W^{s,p(\cdot,\cdot)}(\Omega ).
        \end{equation*}
        That is, the space $W^{s,p(\cdot,\cdot)}(\Omega )$ is continuously embedded in 
        $L^{q(\cdot)}(\partial \Omega ).$ Moreover, this embedding is
        compact.
    \end{theorem}

    \begin{remark}
        {\rm Observe that if $p$ is a continuous bounded variable exponent in $\overline{%
        \Omega }$ and we extend $p$ to $\overline{\Omega }\times \overline{\Omega }$
        as $p(x,y)\coloneqq\frac{p(x)+p(y)}{2}$, then $p^{\ast } (x)$ 
        coincides with the classical Sobolev trace exponent associated with $p(x)$.}
    \end{remark}
    
    \begin{remark} {\rm
	    We also want to observe that Theorem \ref{sobolev.trace} is still holds 
	    if we replace the continuity hypotheses with the assumption that there is
	    $\varepsilon>0$ such that 
	    \[
	           p^\star(x)-\varepsilon>q(x) \quad \text{ in }
            \partial \Omega \cap \{x\in\overline{\Omega}\colon n-s\bar{p}(x) >0\}.
	    \]
	    }
    \end{remark}
    As a simple application of our trace theorem we can mention the following:
    For the local case we have that the Neumann problem
    $$
    \left\{
        \begin{array}{ll}
            -\Delta_{p(x)} u (x) + |u|^{p(x)-2} u (x)  =0  
            \qquad & \mbox{ in } \Omega, \\[10pt]
            \displaystyle |\nabla u(x) |^{p(x)-2} \frac{\partial u }{\partial \eta}(x)  
            = g (x) \qquad & \mbox{ on } \partial \Omega, 
        \end{array}
        \right.
    $$
    can be solved minimizing the functional
    $$
        F (u) \coloneqq
        q \int_\Omega \frac{|\nabla u(x) |^{p(x)}}{p(x)}\, dx + \int_\Omega 
        \frac{|u (x) |^{p(x)}}{p(x)} \, dx -  \int_{\partial \Omega } g (x)u (x)
        \, d\sigma.
     $$ 
    Here $\Delta_{p(x)}u=\mbox{div}\left(|\nabla u|^{p(x)-2}\nabla u\right)$ is the
    $p(x)-$Laplacian and $\dfrac{\partial}{\partial\eta}$ is the outer normal derivative.

    Here we show the following result that is analogous to the one that holds for the local 
    case.

    \begin{theorem}\label{teo.aplic}  
        Let $r\colon\partial \Omega \to(1,\infty )$   be a continuous function such that
        $1<r_{-}\leq r_{+}< +\infty.$  
        If $p$ is symmetric (i.e. $p(x,y)=p(y,x)$) and 
        \[ 
            p^\star(x)>\frac{r(x)}{r(x)-1}
            \quad \text{ in }
                \partial \Omega \cap \{x\in\overline{\Omega}\colon n-sp(x,x) >0\}, 
        \] 
        then, for any  $g\in L^{r(\cdot)}(\partial \Omega )$,
        there exists a unique minimizer of the functional
        $$
            G(u)\coloneqq \int_{\Omega}\int_{\Omega }
            \frac{|u(x)-u(y)|^{p(x,y)}}{p(x,y) |x-y|^{n+sp(x,y)}} \,  dx \, dy  
            + \int_{\Omega } \frac{|u (x)|^{p(x,x)}}{p(x,x)} \, dx -  
            \int_{\partial \Omega } g (x)u (x)\, d\sigma
        $$
        in $W^{s,p(\cdot,\cdot)}(\Omega)$ that verifies
        $$
        \begin{array}{l}
                \displaystyle
                \int_{\Omega}\int_{\Omega }\frac{|u(x)-u(y)|^{p(x,y)-2} (u(x)-u(y)) 
            (\varphi (x) -\varphi(y))}{|x-y|^{n+sp(x,y)}} \,  dx \, dy  \\[10pt]
            \qquad \displaystyle + \int_{\Omega } |u (x)|^{p(x,x)-2} u(x) \varphi (x) \, 
            dx -  \int_{\partial \Omega } g (x)\varphi (x)\, d\sigma =0 
        \end{array}
        $$
        for every $\varphi \in C^{1}$.
    \end{theorem}

    The rest of paper is organized as follows: in the next section, 
    Section \ref{sect-prelim}, we include as preliminaries the statements 
    of known results that will be used in the proof of our main result; while
    in Section \ref{sect-Sob-exp} we include the proof of Theorem \ref{sobolev.trace};
    Finally, in Section \ref{AnAp} we prove Theorem \ref{teo.aplic}.

\section{Preliminaries} \label{sect-prelim}

    In this section we collect some well known results. 
    
    We begin by observing that if  $(A, \Sigma,\mu)$ is a $\sigma-$finite complete space, 
    and $p$ is bounded variable exponent then $f\in L^{p(\cdot)}(A,\mu)$ if only if 
    \[
        \int_A |f(x)|^{p(x)} d\mu(x)<\infty.
    \] 
    
    Our first result in this section is the well known
    Holder's inequality for variable exponents, see \cite[Lemma 3.2.20]{diening}.

    \begin{theorem}[Holder's inequality]
        Let $(A, \Sigma,\mu)$ be a $\sigma-$finite complete space, and $p,q$ and $r$
        be bounded variable exponent such that 
        \label{Holder} 
        \[
            \frac{1}{r(x)}=\frac{1}{p(x)}+\frac{1}{q(x)}
         \]   
         for $\mu-$a.e. $x\in A.$ 
         If $f\in L^{p(\cdot)}(A,\mu)$ and $g\in L^{q(\cdot)}(A,\mu)$, then 
         $fg\in L^{r(\cdot)}(A,\mu)$
         and there is a positive constant $C$ such that
        \begin{equation*}
               \Vert fg\Vert _{\scriptstyle L^{r(\cdot)}(A)}\leq C
                \Vert f\Vert _{\scriptstyle L^{p(\cdot)}(A)}
                \Vert g\Vert_{\scriptstyle L^{q(\cdot)}(A)}.
        \end{equation*}
    \end{theorem}
    
    Our second result is an embedding result.
    
    \begin{theorem}\label{theorm:embed1}
       Let $\Omega \subset \mathbb{R}^{n}$ be a smooth bounded domain, $s\in(0,1)$ and
       $p$ be a  bounded variable exponent such that $p_->1$. 
       If $t\in (0,s)$ and $r\in(1,p_-)$ then
       the space $ W^{s,p(\cdot,\cdot)}(\Omega)$ is continuously embedded in 
       $W^{t,r}(\Omega).$ In addition, there is a positive constant 
       $C=C(p_-,p_+,r,s,t,N,\Omega)$ such that
       \[
            \Vert f\Vert_{\scriptstyle L^r(\Omega)}\le
            C\Vert f\Vert_{\scriptstyle L^{\bar{p}(\cdot)}(\Omega)}\quad
            \text{ and }\quad  \lbrack f\rbrack_{t,r}\le C \lbrack f
            \rbrack_{s,p(\cdot,\cdot)}
       \]
    \end{theorem}
    
    \begin{proof}
	    Given $f\in  W^{s,p(\cdot,\cdot)}(\Omega)$ we want to show that 
	    $f\in  W^{t,r}(\Omega).$
	    Since $r\in(1,p_-),$ $f\in L^{\bar{p}(\cdot)}(\Omega)$ and $\Omega$
	    has finite measure,  by Holder's inequality, we have that $f\in L^r(\Omega).$
	    Then, we only need to show that 
	    \[
	        F_r(x,y)=
	        \dfrac{|f(x)-f(y)|}{|x-y|^{\nicefrac Nr + t}}\in L^r(\Omega\times\Omega).
	    \]
	    Observe that
	    \[ 
	        F_r(x,y)=\dfrac{|f(x)-f(y)|}{|x-y|^{\nicefrac{N}{p(x,y)} + s}}
	        \dfrac{|x-y|^{s-t}}{|x-y|^{N\nicefrac{(p(x,y)-r)}{rp(x,y)}}}
	        :=H(x,y)G(x,y).
	    \]
	    Since $f\in  W^{s,p(\cdot,\cdot)}(\Omega),$ we have that $H\in 
	    L^{p(\cdot,\cdot)}(\Omega\times\Omega).$ Then, by the Holder's inequality
	    we only need to show that $G\in 
	    L^{q(\cdot,\cdot)}(\Omega\times\Omega),$ where 
	    $$q(x,y)=\frac{p(x,y)r}{(p(x,y)-r)}.$$
	    That is, it is enough to show that
	    \begin{equation}\label{eq:bounded}
	         \int_\Omega\int_\Omega\dfrac{|x-y|^{(s-t)
	        \nicefrac{rp(x,y)}{(p(x,y)-r)}}}{|x-y|^N}
	         dx dy<+\infty.
        \end{equation}
	    To prove this, we set $d=\sup\{|x-y|\colon (x,y)\in \Omega\times\Omega\}.$ 
	    Observe that 
	    \begin{align*}
	        d^{(s-t)
	        \nicefrac{rp(x,y)}{(p(x,y)-r)}}&\le \max\{
	        d^{(s-t)
	        \nicefrac{rp_-}{(p_- -r)}},d^{(s-t)
	        \nicefrac{rp_+}{(p_+ -r)}}\},\\
	        \left(\dfrac{|x-y|}{d}\right)^{(s-t)
	        \nicefrac{rp(x,y)}{(p(x,y)-r)}}&\le\left(\dfrac{|x-y|}{d}\right)^{(s-t)
	        \nicefrac{rp_+}{(p_+ -r)}}.
        \end{align*} 
        Then, there is a positive constant $C=C(p_+,p_-,r, s, t,d)$ such that
        \[
            \int_\Omega\int_\Omega\dfrac{|x-y|^{(s-t)
	        \nicefrac{rp(x,y)}{(p(x,y)-r)}}}{|x-y|^N}
	         dx dy\le C \int_\Omega\int_\Omega\dfrac{|x-y|^{(s-t)
	        \nicefrac{rp_+}{(p_+-r)}}}{|x-y|^N}.
	         dx dy
        \] 
        Therefore, since $(s-t)\nicefrac{rp_+}{(p_+-r)}>0$ and $\Omega$ is bounded, 
        we have that \eqref{eq:bounded} holds.
    \end{proof}

    Finally, we recall that in the constant exponent case we have the following 
    fractional Sobolev trace embedding theorem. For the proof we refer to \cite{Gris}.

    \begin{theorem}
        \label{2}
        Let $\Omega \subset \mathbb{R}^{n}$ be an smooth bounded domain,  
        $0<s<1$ and $p\in \lbrack 1,+\infty )$ such that $1<sp<n$.
        Then there exists a positive constant $C=C(n,p,q,s,\Omega )$ such that, for
        any $f\in W^{s,p}(\Omega )$, we have 
        \begin{equation*}
            \Vert f\Vert_{L^{q}(\partial \Omega )}\leq C\Vert f\Vert_{W^{s,p}(\Omega )}
        \end{equation*}
        for any $q$ such that
        \begin{equation*}
            1\leq q \leq \frac{(n-1)p}{n - sp};
          \end{equation*}
        i.e., the space $W^{s,p}(\Omega )$ is
        continuously embedded in $L^{q}(\partial \Omega )$.
        Moreover, this embedding is compact for $q\in [1,\frac{(n-1)p}{n - sp}).$
    \end{theorem}
    
    \begin{remark}{\rm
	   Let $\Omega \subset \mathbb{R}^{n}$ be a smooth bounded domain, $0<s<1$ and
       $p$ be a bounded variable exponent such that $n>sp_->1$. 
       Then there exist $t\in(0,s)$ and $r\in(1,p_-)$ such that $tr\in(1,n).$
       Therefore, by Theorems \ref{theorm:embed1}
       and \ref{2}, we have that  $W^{s,p(\cdot,\cdot)}(\Omega)$ is continuously embedded in
       $L^q(\partial\Omega)$ for all $q\in [1,\frac{(n-1)r}{n - tr}].$ That is, 
       for any $u\in W^{s,p(\cdot,\cdot)}(\Omega),$ $u|_{\partial\Omega}$
       is well defined.}
    \end{remark}
    
\section{The trace theorem}
\label{sect-Sob-exp}

Let us proceed with the proof of Theorem \ref{sobolev.trace}.

\begin{proof}[Proof of Theorem \ref{sobolev.trace}]
    Being $p$ and $q$ continuous, and $\partial \Omega$ compact, 
    from our assumption \eqref{cota.q} we get that there exists a
    positive constant $k$ such that 
    \begin{equation}\label{1.1}
        p^\star(x)-q(x)=\frac{(n-1)p(x,x)}{n-sp(x,x)}-q(x)\geq k>0,	
    \end{equation}
    for every $x\in\partial \Omega$ (here $p^\star$ is understood as $+\infty$ when
    $n-sp(x,x) \leq 0$).

    Since $p$ and $q$ are continuous, 
    using  \eqref{1.1} we can find a constant 
    $\epsilon =\epsilon(p,q,k,s)$ and a finite family of open sets 
    $B_{i} \subset \overline{\Omega}$ such that 
    \begin{equation*}
        \partial \Omega =\bigcup _{i=1}^{N}\overline{B_{i}} \cap \partial \Omega,
         \qquad diam(B_{i})\coloneqq\sup\{|x-y|\colon (x,y)\in B_i\times B_i\}<\epsilon ,
    \end{equation*}
    and 
    \begin{equation}
        \displaystyle\frac{(n-1) p(z,y)}{n-sp(z,y)}-q(x)\geq \frac{k}{2}
         \label{1.2}
    \end{equation}
    for every $x\in \partial \Omega \cap \overline{B_{i}}$ and $(z,y)\in \overline{B_{i}}
    \times \overline{B_{i}}$ (here we set again $\frac{(n-1) p(z,y)}{n-sp(z,y)}$ as $+\infty$
     when $n-sp(z,y) \leq 0$).

    Given $\delta>0$ small we can select
    \begin{equation*}
        p_{i} < \inf\{p(z,y)-\delta \colon (z,y)\in B_{i}\times B_{i}\}\qquad 1<s_i p_i<n,
    \end{equation*}
    such that
    \begin{equation}
        \displaystyle\frac{\displaystyle (n-1)p_{i}}{n-s_i\displaystyle p_{i}}\displaystyle
        \geq \frac{k}{3}+q(x)  \label{1.3}
    \end{equation}
    for each $x\in \overline{B_{i}} \cap \partial \Omega$. 
    We can choose $\delta$ smaller is necessary in order to have $$p_{i}-1>\delta >0.$$

    Hence, by Theorem \ref{theorm:embed1} and the trace theorem for constant exponents 
    (see Theorem \ref{2}), we obtain the existence of a constant 
    $C=C(n, p_{i}, s_i, \epsilon, B_i)$ such that 
    \begin{align}  \label{1.3.0}
        \Vert f\Vert_{L^{p^{*}_{i}}(\overline{B_{i}}\cap \partial \Omega)}
        \leq C \Big(\Vert f\Vert_{L^{p_{i}}(B_{i})}+[f]_{t,p_{i}}(B_{i}) \Big).
    \end{align}

    Now we want to show that the following three statements hold.

    \begin{itemize}
        \item[(A)] There exists a constant $c_{1}$ such that 
            \begin{align*}
                \sum_{i=0}^{N} 
                \Vert f\Vert_{\scriptstyle L^{p^{*}_{i}}(\overline{B_{i}}\cap 
                \partial \Omega )}
                \geq c_{1} \Vert f\Vert_{\scriptstyle L^{q(\cdot)}(\partial \Omega)}.
            \end{align*}

        \item[(B)] There exists a constant $c_{2}$ such that 
            \begin{align*}
                \sum_{i=0}^{N} \Vert f\Vert_{\scriptstyle L^{p_{i}}(B_{i})} 
                \leq c_{2} \Vert f\Vert_{\scriptstyle L^{\bar{p}(\cdot)}(\Omega)}.
            \end{align*}

        \item[(C)] There exists a constant $c_{3}$ such that 
            \begin{align*}
                    \sum_{i=0}^{N}\lbrack f\rbrack_{s_i,p_{i}}(B_{i}) 
                    \leq c_{3}\lbrack f\rbrack _{s,p(\cdot,\cdot)}(\Omega).
             \end{align*}
    \end{itemize}

    These three inequalities and \eqref{1.3.0} give
        \begin{align*}
	        \Vert f\Vert _{\scriptstyle L^{q(\cdot)}(\partial \Omega )}
	        &\leq C\sum_{i=0}^{N}\Vert f\Vert _{L^{p_{i}^{\ast }}(\overline{B_{i}} 
	        \cap \partial \Omega)} \\
	        &\leq C\sum_{i=0}^{N}\left (\Vert f
	        \Vert_{\scriptstyle L^{p_{i}}(B_{i})}+[f]_{s_i,p_{i}}(B_{i})\right) \\
	        &\leq C\left(\Vert f\Vert _{\scriptstyle L^{\bar{p}(\cdot)}(\Omega)}
	        +[f]_{s,p(\cdot,\cdot)}(\Omega )\right) \\
	        &=C\Vert f\Vert _{s,p(\cdot,\cdot)},
        \end{align*}
      as we wanted to show.

Therefore, we have to show $(A)$, $(B)$ and $(C)$.
Let us start with $(A)$. For $x\in \partial \Omega$ we have 
\begin{equation*}
|f(x)|\leq \sum_{i=0}^{N}|f(x)|\chi_{\overline{B_{i}}} (x).
\end{equation*}
Hence 
\begin{equation}
\Vert f\Vert_{\scriptstyle L^{q(\cdot)}(\partial \Omega )}\leq \sum_{i=0}^{N}\Vert 
f\Vert_{\scriptstyle L^{q(\cdot)}(\overline{B_{i}}\cap \partial \Omega)},
\label{besi}
\end{equation}
Since we have
$$
 \frac{\displaystyle (n-1)p_{i}}{n-t\displaystyle p_{i}}\displaystyle
\geq \frac{k_{1}}{3}+q(x) > q(x) 
$$
we can take $a_{i}(x)$ such that 
\begin{equation*}
\frac{1}{q(x)}= \frac{ n-t\displaystyle p_{i}  }{\displaystyle (n-1)p_{i}}+\frac{1}{a_i(x)}.
\end{equation*}
Using Theorem \ref{Holder} we obtain 
\begin{align*}
\Vert f\Vert _{L^{q(x)}(\overline{B_{i}}\cap \partial \Omega)}& \leq c\Vert f\Vert_{L^{\frac{ (n-1)p_{i}}{n-t p_{i}}}(\overline{B_{i}}\cap \partial \Omega)}\Vert 1\Vert _{L^{a_{i}(x)}(\overline{B_{i}}\cap \partial \Omega)} \\
& \leq C\Vert f\Vert _{L^{\frac{ (n-1)p_{i}}{n-t p_{i}}}(\overline{B_{i}}\cap \partial \Omega)}.
\end{align*}
Thus, we get $(A)$.

To show $(B)$ we argue in a similar way using that $p(x,x)>p_{i}$ for $x\in
B_{i}$.

In order to prove $(C)$ let us set 
    \begin{equation*}
        F(x,y)\coloneqq\frac{|f(x)-f(y)|}{|x-y|^{s}},
    \end{equation*}%
and observe that 
\begin{align}
    \lbrack f]_{t,p_{i}}(B_{i})& =\left( \int_{B_{i}}\int_{B_{i}}\frac{%
    |f(x)-f(y)|^{p_{i}}}{|x-y|^{n+tp_{i}+sp_{i}-sp_{i}}}\,dxdy\right) ^{\frac{1}{%
    p_{i}}}  \notag \\
    & =\left( \int_{B_{i}}\int_{B_{i}}\left( \frac{|f(x)-f(y)|}{|x-y|^{s}}%
    \right)^{p_{i}}\,\frac{dxdy}{|x-y|^{n+(t-s)p_{i}}}\right) ^{\frac{1}{p_{i}}}
    \notag \\
    & =\Vert F\Vert_{\scriptstyle  L^{p_{i}}(\mu ,B_{i}\times B_{i}})  \label{1.3.1} \\
    & \leq C \Vert F\Vert_{\scriptstyle L^{p(\cdot,\cdot)}(\mu ,B_{i}
    \times B_{i})}\Vert 1\Vert_{\scriptstyle  L^{b_{i}(\cdot,\cdot)}
    (\mu,B_{i}\times B_{i})}  \notag \\
    & \leq C\Vert F\Vert_{\scriptstyle  L^{p(\cdot,\cdot)}(B_{i}\times B_{i},\mu)},  \notag
\end{align}
where we have used Theorem \ref{Holder} with 
\begin{equation*}
\frac{1}{p_{i}}=\frac{1}{p(x,y)}+\frac{1}{b_{i}(x,y)},
\end{equation*}
but considering the measure in $B_{i}\times B_{i}$ given by 
\begin{equation*}
d\mu (x,y)=\frac{dxdy}{|x-y|^{n+(t-s)p_{i}}}.
\end{equation*}

Now our aim is to show that 
\begin{equation}
\Vert F\Vert _{\scriptstyle L^{p(\cdot,\cdot)}(B_{i}\times B_{i},\mu)}\leq 
C[f]_{s,p(\cdot,\cdot)}(B_{i})
\label{seis}
\end{equation}
for every $i$. If this holds, then we immediately get $(C)$ using that for every $i$ it holds that
$$ [f]_{s,p(\cdot,\cdot)}(B_{i})
 \leq [f]_{s,p(\cdot,\cdot)}(\Omega).$$

    Set $\lambda=\lbrack f]_{s,p(\cdot,\cdot)}(B_i ).$
    Then, using that $diam (B_i) <\epsilon <1,$  we get $|x-y|<1$ for every 
    $(x,y) \in B_{i} \times B_{i}$ and hence
    \begin{align*}
        & \int_{B_{i}}\int_{B_{i}}\left( \frac{|f(x)-f(y)|}{\lambda
        |x-y|^{s}}\right) ^{p(x,y)}\frac{dxdy}{|x-y|^{n+(t-s)p_{i}}} \\
        & \qquad =\int_{B_{i}}\int_{B_{i}} |x-y|^{(s-t)p_{i}}\frac{
        |f(x)-f(y)|^{p(x,y)}}{\lambda ^{p(x,y)}|x-y|^{n+sp(x,y)}}\,dxdy \\
        & \qquad < \int_{B_{i}}\int_{B_{i}}
        \frac{|f(x)-f(y)|^{p(x,y)}}{\lambda^{p(x,y)}|x-y|^{n+sp(x,y)}}\,dxdy\le1.
    \end{align*}
    Therefore 
    \[
        \Vert F\Vert _{L^{p(\cdot,\cdot)}(B_{i}\times B_{i},\mu)}\leq 
        \lambda=[f]_{s,p(\cdot,\cdot)}(B_{i}),
    \]
    which implies the desired inequality.

    Finally, we recall that the previous embedding is compact since in the
    constant exponent case we have that for subcritical exponents the embedding
    is compact. Hence, for a bounded sequence in $W^{s,p(\cdot,\cdot)}(\Omega)$, 
    $f_i$, we can mimic the
    previous proof obtaining that for each $B_i$ we can extract a convergent
    subsequence in $L^{q(\cdot)} (B_i\cap \partial \Omega)$.
\end{proof}

\begin{remark}
{\rm Our result is sharp in the following sense: if 
\begin{equation*}
\frac{(n-1) p(x_{0},x_{0})}{n-sp(x_{0},x_{0})}<q(x_{0})
\end{equation*}
for some $x_{0}\in \partial \Omega $, then the embedding of $W^{s,p(\cdot,\cdot)}(\Omega)$ in $L^{q(\cdot)}(\partial\Omega )$
cannot hold. In fact, from our continuity conditions on $p$
and $q$ there is a small ball $B_{\delta }(x_{0})$ such that 
\begin{equation*}
\max_{\overline{B}_{\delta }(x_{0})\times \overline{B}_{\delta }(x_{0})}
\frac{(n-1) p(x,y)}{n-sp(x,y)}<\min_{\overline{B}_{\delta }(x_{0}) \cap \partial \Omega}q(x).
\end{equation*}
In this situation, with the same arguments that hold for the constant
exponent case, one can find a sequence $f_{k}$ supported inside $B_{\delta
}(x_{0})$ such that $\Vert f_{k}\Vert _{s,p(\cdot,\cdot)}\leq C$ and $\Vert f_{k}\Vert
_{L^{q(\cdot)}(\overline{B_{\delta }(x_{0})} \cap \partial \Omega}\rightarrow +\infty $. In fact, we just consider a smooth,
compactly supported function $g$ and take $$f_k(x) = k^a g(k(x-x_0))$$ with $a$ such that
$ a p(y,z) -n + sp(y,z) \leq 0 $ and $a q(x) -(n-1) >0$ for $x\in \overline{B}_{\delta }(x_{0}) \cap \partial \Omega$ and $y,z \in \overline{B}_{\delta }(x_{0})$.}

{\rm Finally, we mention that the critical case 
\begin{equation*}
\frac{(n-1) p(x,x)}{n-sp(x,x)}\geq q(x)
\end{equation*}
with equality for some $x_{0}\in \partial\Omega $ is left open. }
\end{remark}

\begin{remark} \label{rem.ttt}
{\rm We observe that with the same arguments we can also deal with
variable $s$ in the fractional Sobolev space. That is, given 
$$
s:\overline{\Omega }\times \overline{\Omega }\rightarrow
(0,1)
$$
a symmetric and continuous function and $p$ as before,
we can consider
the seminorm
\begin{equation*}
\lbrack f]_{s(\cdot,\cdot),p(\cdot,\cdot)}(\Omega )\coloneqq\inf \left\{ \lambda >0\colon\int_{\Omega
}\int_{\Omega }\frac{|f(x)-f(y)|^{p(x,y)}}{\lambda ^{p(x,y)}|x-y|^{n+s(x,y)p(x,y)}
}<1\right\}
\end{equation*}
and, as before, the norm 
\begin{equation*}
\Vert f\Vert _{s(\cdot,\cdot), p(\cdot,\cdot)}\coloneqq
\Vert f\Vert _{L^{\bar{p}(\cdot}(\Omega )}+[f]_{s(\cdot,\cdot),p(\cdot,\cdot)}(\Omega ).
\end{equation*}
In this case, we have that
\begin{equation} \label{cota.q.44}
\frac{(n-1)p(x,x)}{n-s(x,x)p(x,x)}>q(x),
\end{equation}
for $x\in \partial \Omega \cap \{ n-s(x,x)p(x,x) >0\}$ implies the existence of a constant 
$C$ such that 
\begin{equation*}
\Vert f\Vert _{L^{q(\cdot)}(\partial \Omega )}\leq 
C\Vert f\Vert _{W^{s(\cdot,\cdot), p(\cdot,\cdot)} (\Omega)}.
\end{equation*}
}
\end{remark}

\begin{remark} {\rm We also have a Sobolev-Sobolev trace embedding. 
Using that for constant $p$ one has the embedding
$$
W^{s,p} (B) \rightarrow W^{s-\frac1p ,p} (\partial B)
$$
(see \cite{Gris}) one can show (arguing exactly as before) the following result: let  
$$
s\colon\overline{\Omega }\times \overline{\Omega }\rightarrow
(0,1), \qquad p\colon\overline{\Omega }\times \overline{\Omega }\rightarrow
(1,\infty)
$$
and 
$$
t\colon\partial \Omega \times \partial \Omega \rightarrow
(0,1), \qquad q\colon\partial \Omega \times \partial \Omega \rightarrow
(1,\infty)
$$
be continuous functions with 
$$
\bar{p}(x)\bar{s}(x)\coloneqq p(x,x) s(x,x) >1
$$
and
$$
\bar{t}(x)\coloneqq t(x,x) < \bar{s}(x)-\frac{1}{\bar{p}(x) } \qquad 
\mbox{ and } \qquad \bar{q}(x)\coloneqq q(x,x) < \bar{p}(x)
$$
for every $x\in \partial \Omega$.
Then it holds that
$$
W^{s(\cdot,\cdot),p(\cdot,\cdot)} (\Omega) \rightarrow 
W^{t(\cdot,\cdot),q(\cdot,\cdot)} (\partial \Omega).
$$
Notice that here we let (as the notation suggests)
\begin{equation*}
\lbrack f]_{t(\cdot,\cdot),q(\cdot,\cdot)}(\partial \Omega )
\coloneqq\inf \left\{ \lambda >0\colon\int_{\partial\Omega
}\int_{\partial\Omega }\frac{|f(x)-f(y)|^{q(x,y)}}{\lambda^{q(x,y)}|x-y|^{n+t(x,y)q(x,y)}
}d\sigma d\sigma<1\right\}
\end{equation*}
and the norm 
\begin{equation*}
\Vert f\Vert _{W^{t(\cdot,\cdot), q(\cdot,\cdot)}(\partial\Omega)}\coloneqq
\Vert f\Vert _{L^{\bar{q}(\cdot)}(\partial \Omega )}
+[f]_{t(\cdot,\cdot),q(\cdot,\cdot)}(\partial \Omega ).
\end{equation*}

In fact, to prove this result, one first observe that the trace theorem in its Sobolev-Lebesgue version 
gives that
\begin{equation*}
\Vert f\Vert _{L^{\bar{q}(\cdot)}(\partial \Omega )}\leq C\Vert f\Vert_{{s(\cdot,\cdot), 
p(\cdot,\cdot)}},
\end{equation*}
see Remark \ref{rem.ttt} (notice that we have $\bar{q}(x) < \bar{p}(x) < 
\frac{(n-1)\bar{p}(x)}{n-\bar{s}(x)\bar{p}(x)}$ for $x\in \partial 
\Omega$ due to the fact that we assumed $\bar{p}(x)\bar{s}(x) >1$). Hence we are left with the proof of an inequality of the form 
$$
[f]_{t(\cdot,\cdot),q(\cdot,\cdot)}(\partial \Omega ) 
\leq C\Vert f\Vert _{s(\cdot,\cdot), p(\cdot,\cdot)}.
$$
Here one can mimic the same proof as in the Sobolev-Lebesgue trace theorem using that there exist
a finite number of
sets $B_i$ such that $\cup_i^N \overline{B_i}$ cover $\partial \Omega$ 
and constant exponents $s_i,t_i$, $p_i,q_i$ such that
$$
p(x,y) >p_i, \qquad  s_i > s(x,y), \qquad \forall x,y \in \overline{B_i} \times \overline{B_i}
$$
$$
q(x,y) < q_i \qquad
t_i <t(x,y),  \qquad \forall x,y \in \overline{B_i}\cap \partial \Omega \times \overline{B_i}
\cap \partial \Omega
$$
with 
$$
p_i s_i >1, \qquad
t_i < s_i -\frac{1}{p_i} \qquad \mbox{ and } \qquad q_i < p_i
$$
and then use the Sobolev-Sobolev trace theorem with constant exponents
$$
W^{s_i,p_i} (B) \rightarrow W^{t_i ,q_i} (\partial B),
$$
add over $i$ and conclude as before.
}
\end{remark}

\section{An application}\label{AnAp}

Now we turn our attention to the proof of Theorem \ref{teo.aplic}.

\begin{proof}[Proof of Theorem \ref{teo.aplic}]
We just observe that we can apply the direct method of calculus of
variations. Note that the functional $G$ is
strictly convex (this holds since for any $x$ and $y$ the
function $t\mapsto t^{p(x,y)}$ is strictly convex) and weakly lower semicontinuous.

From our previous results, $W^{s,p(\cdot,\cdot)}(\Omega)$ is compactly
embedded in $L^{q(\cdot)}(\partial \Omega)$ for $q(x)<p^{*}(x)$, see Theorem \ref{sobolev.trace}. 
In particular, we have that $W^{s,p(\cdot,\cdot)}(\Omega)$ is compactly
embedded in $L^{\frac{r(\cdot)}{r(\cdot)-1}}(\partial\Omega)$.

Let us see that $G$ is coercive. We have 
\begin{align*}
G(u)& =\int_{\Omega }\int_{\Omega }\frac{|u(x)-u(y)|^{p(x,y)}}{
|x-y|^{n+sp(x,y)}p(x,y)}dxdy+\int_{\Omega }\frac{|u(x)|^{p(x,x)}}{p(x,x)}
dx-\int_{\partial\Omega }g(x)u(x)\,d\sigma \\
& \geq \int_{\Omega }\int_{\Omega }\frac{|u(x)-u(y)|^{p(x,y)}}{
|x-y|^{n+sp(x,y)}p(x,y)}dxdy+\int_{\Omega }\frac{|u(x)|^{p(x,x)}}{p(x,x)}dx \\
& \qquad -\Vert g\Vert _{L^{r(x)}(\partial \Omega )}\Vert u\Vert _{L^{\frac{r(x)}{
r(x)-1}}(\partial \Omega )} \\
& \geq \int_{\Omega }\int_{\Omega }\frac{|u(x)-u(y)|^{p(x,y)}}{
|x-y|^{n+sp(x,y)}p(x,y)}dxdy+\int_{\Omega }\frac{|u(x)|^{p(x,x)}}{p(x,x)}
dx-C\Vert u\Vert_{s,p(\cdot,\cdot)}.
\end{align*}

Now, we choose a sequence $u_{j}$ such that 
$\Vert u_{j}\Vert_{s, p(\cdot,\cdot)}\rightarrow \infty $ as $j\rightarrow \infty $. 
Let us assume that $\Vert u_j\Vert_{s, p(\cdot,\cdot)}>1$. Then we have 
\begin{align*}
\frac{G(u_j)}{\Vert u_j\Vert_{s, p(\cdot,\cdot)}}& \geq \frac{1}{\Vert u_j\Vert_{s, p(\cdot,\cdot)}}
\left( \int_{\Omega }\int_{\Omega }\frac{|u_j(x)-u_j(y)|^{p(x,y)}}{
|x-y|^{n+sp(x,y)}p(x,y)}dxdy+\int_{\Omega }\frac{|u_j(x)|^{p(x,x)}}{p(x,x)}
dx\right) -C \\
& \geq \Vert u_j\Vert_{s, p(\cdot,\cdot)}^{p_{-}-1}-C.
\end{align*}
Then we obtain 
\begin{equation*}
G(u_{j})\geq \Vert u_{j}\Vert_{s, p(\cdot,\cdot)}^{p_{-}-1}-C\Vert
u_{j}\Vert_{s, p(\cdot,\cdot)}\rightarrow \infty ,
\end{equation*}
and we conclude that $G$ is coercive. Therefore, there is a unique
minimizer of $G$ in $W^{s,p(\cdot,\cdot)}(\Omega)$.

To show that it holds that the minimizer verifies 
$$
\begin{array}{l}
\displaystyle
 \int_{\Omega
}\int_{\Omega }\frac{|u(x)-u(y)|^{p(x,y)-2} (u(x)-u(y)) (\varphi (x) -\varphi(y))}{|x-y|^{n+sp(x,y)}} 
\,  dx \, dy  \\[10pt]
\qquad \displaystyle + \int_{\Omega } |u (x)|^{p(x,x)-2} u(x) \varphi (x) \, dx -  \int_{\partial \Omega } g (x)\varphi (x)
\, d\sigma =0 
\end{array}
$$
for every $\varphi \in C^{1}$ one just have to differentiate
$$
t \mapsto G(u+tv)
$$
and use that this derivative vanishes at $t=0$ since $u$ is a minimum of $G$.
\end{proof}

\end{document}